\documentclass[11pt, reqno]{amsart}

\usepackage{amsmath, amsfonts, amssymb, amsbsy, bigstrut, graphicx, enumerate,  upref, longtable, comment, comment}

\usepackage{pstricks}

\usepackage[breaklinks]{hyperref}

\usepackage[numbers, sort&compress]{natbib}

\newcommand{\ep}{\epsilon}

\newtheorem{thm}{Theorem}[section]
\newtheorem{lmm}[thm]{Lemma}

\newtheorem{prop}[thm]{Proposition}
\newtheorem{defn}[thm]{Definition}

\newtheorem{conj}[thm]{Conjecture}
\theoremstyle{definition}

\newcommand{\ee}{\mathbb{E}}

\newcommand{\mf}{\mathcal{F}}

\newcommand{\cp}{\mathcal{P}}

\newcommand{\pp}{\mathbb{P}}

\newcommand{\rr}{\mathbb{R}}

\newcommand{\var}{\mathrm{Var}}

\newcommand{\zz}{\mathbb{Z}}

\newcommand{\fpar}[2]{\frac{\partial #1}{\partial #2}}

\numberwithin{equation}{section}

\begin{document}
\title[Path localization in directed polymers]{Proof of the path localization conjecture for directed polymers}
\author{Sourav Chatterjee}
\address{\newline Department of Statistics \newline Stanford University\newline Sequoia Hall, 390 Serra Mall \newline Stanford, CA 94305\newline \newline \textup{\tt souravc@stanford.edu}}
\thanks{Research partially supported by NSF grant DMS-1608249}

\begin{abstract}
It is a well-known open problem in the literature on random polymers to show that a directed polymer in random environment localizes around a favorite path at low temperature. A precise statement of this conjecture is formulated and proved in this article.
\end{abstract}

\maketitle


\section{Introduction and main result}
Ever since the seminal work of \citet{anderson58}, the phenomenon of localization in the presence of random impurities has been a recurrent theme in the study of disordered systems in statistical physics. Much of this was initially based on simulation studies and heuristic arguments. Rigorous proofs of localization phenomena started appearing much later, and it is still considered to be a difficult area with many open problems. This paper is about one such problem. Let us begin by defining what is meant by localization in the relevant context.
\subsection{Path localization}\label{localsec}
Fix a dimension $d\ge 1$. This number will remain fixed throughout this article. For each $n$, let $\cp_n$ be the set of all paths of length $n$ starting from the origin in the nearest-neighbor graph on $\zz^d$. A path $p\in \cp_n$ will be denoted as a sequence of $n$ vertices $(x_1,\ldots, x_n)$, omitting the starting point (the origin). For two paths $p=(x_1,\ldots,x_n)$ and $p'=(x_1',\ldots,x_n')$, we define 
\[
p\cap p' := \{k: x_k=x_k'\},
\]
and let $|p\cap p'|$ denote the size of $p\cap p'$. 

Now suppose that $P$ is a random path of length $n$ starting from the origin and $\mu$ is the law of $P$. In other words, $\mu$ is a probability measure on $\cp_n$.  A natural quantification of how much $\mu$ `localizes' around a given path $p$ is given by the quantity
\[
\ell(\mu, p) := \frac{\ee|P\cap p|}{n},
\]
that is, the expected fraction of times $P$ coincides with $p$. If we maximize this over all paths, we get a natural measure of the localization of $\mu$. 
\begin{defn}\label{ldef}
Let $P$ and $\mu$ be as above. We define the degree of localization of $\mu$ (or $P$) as 
\[
\ell(\mu) := \max_{p\in \cp_n} \ell(\mu, p)= \max_{p\in \cp_n} \frac{\ee|P\cap p|}{n}. 
\]
\end{defn}
In other words, $\ell(\mu)\ge \delta$ if and only if there exists a path $p$ such that $\ee|P\cap p|\ge \delta n$. It is not difficult to check that if $P$ is a simple symmetric random walk of length $n$, then
\[
\ell(\mu)\asymp
\begin{cases}
n^{-1/2} &\text{ if } d =1,\\
n^{-1} \log n &\text{ if } d=2,\\
n^{-1} &\text{ if } d\ge 3.
\end{cases}
\]
Consequently, the law of the simple symmetric random walk does not localize as $n\to \infty$ in any dimension. The following definition makes this statement precise.
\begin{defn}\label{localdef}
For each $n$, let $\mu_n$ be a probability measure on $\cp_n$. We will say that the sequence $(\mu_n)_{n\ge 0}$ localizes if 
\[
\liminf_{n\to\infty} \ell(\mu_n) >0.
\]
\end{defn}
In this article we will however be dealing with random probability measures on paths, not deterministic ones. It is not obvious what may be the right way to generalize Definition \ref{localdef} to random measures. We adopt the following definition in this paper.
\begin{defn}\label{randomlocaldef}
Suppose that we have, for each $n$, a random probability measure $\mu_n$ on $\cp_n$. We will say that the sequence $(\mu_n)_{n\ge 1}$ exhibits path localization in the limit if, whenever $\ell(\mu_n)$ converges in law to a random variable $X$ through a subsequence, we have $\pp(X>0)=1$.
\end{defn}
It is not difficult to see that the above definition is equivalent to demanding that for any $\ep>0$ there is some $\delta>0$ such that
\[
 \limsup_{n\to\infty} \pp(\ell(\mu_n)\le \delta) \le \ep.
\] 
Yet another equivalent criterion is that $(1/\ell(\mu_n))_{n\ge 1}$ is a tight family of random variables. 

In the literature, localization of a measure on paths is sometimes measured by the expected overlap, that is, by the quantity
\begin{equation}\label{overlapdef}
\rho(\mu) := \frac{\ee|P\cap P'|}{n},
\end{equation}
where $P$ and $P'$ are independent paths with law $\mu$. The following simple result shows that the measures $\ell(\mu)$ and $\rho(\mu)$ are equivalent for measuring localization.
\begin{prop}\label{equivprop}
For any probability measure $\mu$ on $\cp_n$,
\[
\ell(\mu)^2 \le \rho(\mu)\le \ell(\mu). 
\]
\end{prop}
This result is proved in Section \ref{proofsec}. A consequence of this result is that the criterion for localization in Definition \ref{randomlocaldef} can be equivalently stated  in terms of $\rho(\mu_n)$ instead of $\ell(\mu_n)$. We can use either definition to study localization of random probability measures on paths. Proofs are easier to do with $\rho(\mu_n)$ because it arises naturally from integration by parts (which is in fact what's done in this paper). The definition using $\ell(\mu_n)$ is preferred in this manuscript because it probably appeals better to intuition. 

\subsection{Directed polymers in random environment}
Let $\cp_n$ be as in the previous subsection. Let $\nu$ be a probability measure on $\rr$. Let $(\omega_{k,x})_{k\ge 1, x\in \zz^d}$ be a collection of i.i.d.~random variables with law $\nu$. For each $n\ge 1$ and $\beta \ge0$, let $\mu_{\beta,n}$ be the random probability measure on $\cp_n$ that puts mass proportional to 
\[
\exp\biggl(\beta \sum_{k=1}^n \omega_{k,x_k}\biggr)
\]
on each path $p=(x_1,\ldots,x_n)\in \cp_n$. This defines the law of a  $(d+1)$-dimensional `directed polymer in random environment'. The model is sometimes also known as `random walk in random potential'. The parameter $\beta$ is called the inverse temperature, and $\mu_{\beta,n}$ is called the Gibbs measure or the polymer measure. 

This model was introduced by \citet{husehenley85} as a toy model for studying domain walls of the Ising model in the presence of impurities. It was presented in the mathematical literature as a model of directed polymers in random medium  by \citet{imbriespencer88} and \citet{bolthausen89}. Since then, this and other related models have generated an enormous literature in probability and mathematical physics. See \cite{bateschatterjee16, comets17, giacomin07, denhollander09} for partial surveys.  

A striking feature of the directed polymer model is that when $\beta$ is sufficiently large, the endpoint of the polymer localizes. This is a special instance of the wider phenomenon of localization in the presence of random impurities. The localization of the endpoint of directed polymers was first proved by \citet{cometsshigayoshida03} and \citet{carmonahu06}. A very detailed picture in an exactly solvable case was obtained by \citet{cometsnguyen16}. An abstract study of the endpoint localization phenomenon in a general setting was recently carried out in \citet{bateschatterjee16}, where the set of all possible limiting laws of the endpoint distribution was characterized using an abstract framework. There is also a different line of work on localization of directed polymers in heavy-tailed random environments~\cite{auffingerlouidor11, hamblymartin07, torri16, vargas07}. The challenges are quite different in that scenario, because in a heavy-tailed environment, the very large $\omega$'s are the ones that determine the nature of the Gibbs measure  and essentially determine the most likely polymer path. For the thorough discussion of the heavy-tailed case, see \cite[Chapter 6]{comets17}.

It is a longstanding folklore conjecture, supported by simulations, that it is not only the endpoint of the polymer that localizes at low temperature, but rather the whole path localizes. Let $\ell(\mu_{\beta, n})$ and $\rho(\mu_{\beta,n})$ denote the degree of localization and the expected overlap for the Gibbs measure $\mu_{\beta,n}$, as defined in the previous section. In a Gaussian random environment, it is known that when $\beta$ is large enough, 
\[
\liminf_{n\to\infty} \ee(\rho(\mu_{\beta,n})) >0.
\]
As far as I know, the proof of this result was written down for the first time in \citet[Chapter 6]{comets17}, although versions for slightly different models appeared earlier in the works of \citet{cometscranston13} and \citet{cometsyoshida13}.  In the language of statistical physics, this shows that path localization happens in the averaged sense. 

\subsection{Result}\label{resultsec}
The main result of this article, stated below, is that under certain conditions on the law $\nu$ of the environment, the directed polymer model has the property of path localization in the sense of Definition \ref{randomlocaldef}. In the language of statistical physics, this is one way of saying that path localization happens in the {\it quenched sense} at low temperature. 

The conditions on $\nu$ are as follows.  Suppose that $\nu$ has a probability density $f$ with respect to Lebesgue measure. We will assume that there is a bounded interval $(a,b)$ such that $f$ is nonzero and differentiable in $(a,b)$ and zero outside. Let 
\[
m := \int_a^b x f(x)dx
\]
be the expected value of $\nu$. For $x\in (a,b)$, let
\begin{equation}\label{hdef}
h(x) := \frac{\int_x^b (y-m) f(y)dy}{f(x)}. 
\end{equation}
We will assume that
\begin{equation}\label{hcond}
\sup_{a<x<b} |h'(x)|<\infty.
\end{equation}
This completes the list of all the assumptions that we will make about the law of the environment. It is not difficult to check that the conditions are satisfied for a large class of bounded random variables. For example, the uniform distribution on any bounded interval satisfies these conditions. 
\begin{thm}\label{mainthm}
Consider the model of $(d+1)$-dimensional directed polymers of length $n$ in i.i.d.~random environment. Suppose that the law of the environment satisfies the conditions listed above. Then there exist two positive constants $\beta_0$ and $C$, depending only on the law of environment and the dimension $d$, such that if $\beta\ge \beta_0$, then for any $\delta \in (0,1)$ and any $n$, 
\[
\pp(\ell(\mu_{\beta,n})\le \delta)\le \pp(\rho(\mu_{\beta,n}) \le \delta)\le e^{C\beta} \sqrt{\delta}.
\]
Consequently, the polymer measure exhibits path localization in the sense of Definition \ref{randomlocaldef}.
\end{thm}
Incidentally, the proof technique for Theorem \ref{mainthm} does not immediately generalize to the case of $\nu$ with unbounded support, such as Gaussian. 
\subsection{Simulation results and a conjecture}
To get a better understanding of the behavior of the quenched overlap, some simulations were carried out. The simulation results  indicated that in certain situations, quenched path localization may in fact hold in the following stronger sense. 
\begin{conj}\label{conj1}
Under suitable conditions on the law of the environment, there exists some $\beta_0$ such that for any $\beta \ge \beta_0$, there is some $\delta>0$ such that  
\[
\lim_{n\to \infty} \pp(\ell(\mu_{\beta,n})\ge \delta)=\lim_{n\to \infty} \pp(\rho(\mu_{\beta,n})\ge \delta) =1.
\]
\end{conj}
For example, consider the model of a $(1+1)$-dimensional polymer in Uniform$[-1,1]$ environment, with $n=300$ and $\beta=3$. The model was simulated $1000$ times, and the quenched expectation of the overlap was computed for each simulation. Figure \ref{fig1} shows the histogram of these values. The histogram suggests that the limiting law of $\rho(\mu_{\beta,n})$ is supported on an interval $(a,b)$ where $a>0$ and $b<1$, thus giving evidence in favor of Conjecture \ref{conj1}. The simulation was repeated with other values of $n$ and $\beta$, and it showed similar conclusions on each occasion.

\begin{figure}[t]
\centering
\includegraphics[width = .95\textwidth]{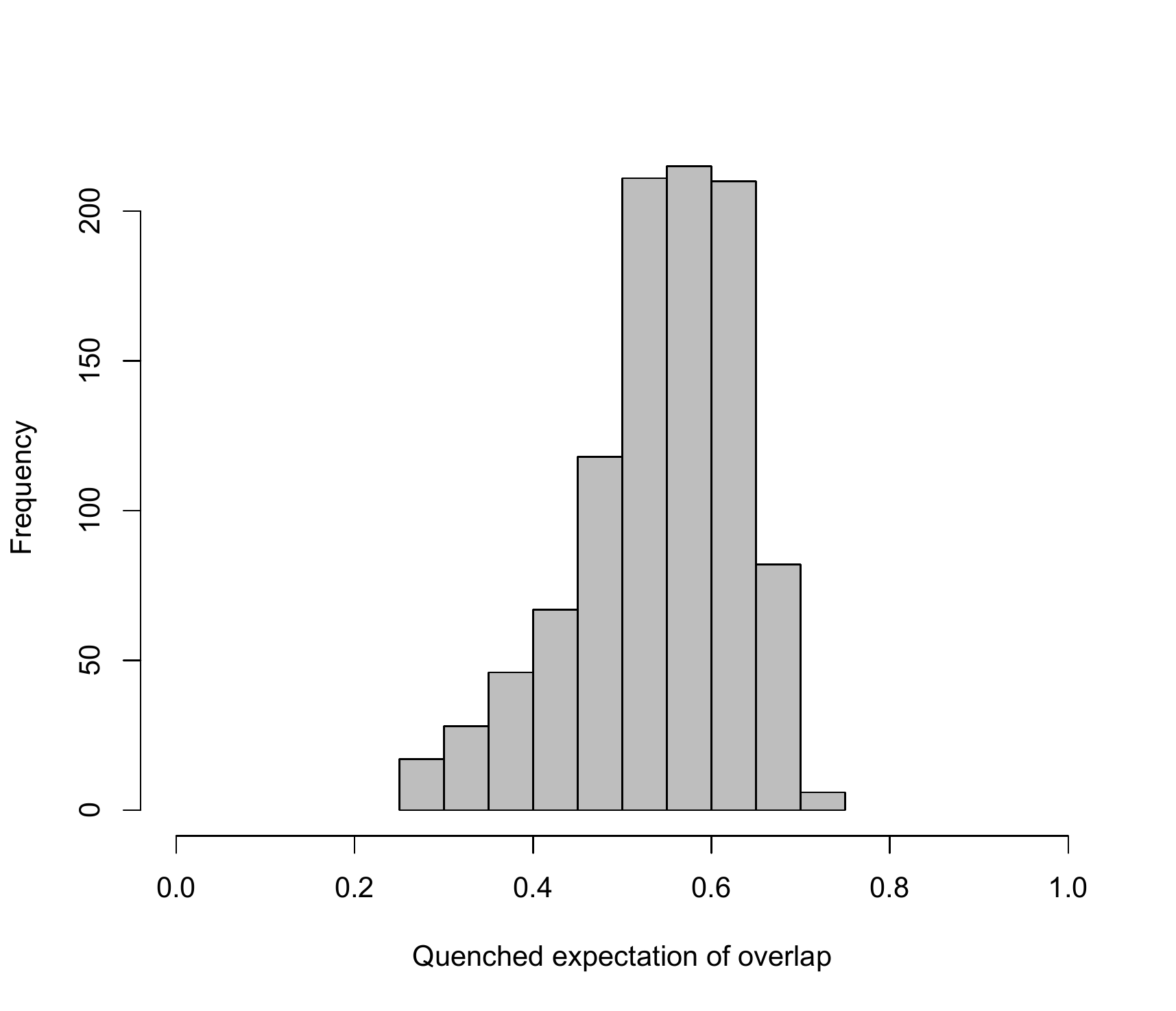}
\caption{Histogram of $1000$ simulated values of $\rho(\mu_{\beta,n})$. Here $d=1$, $n= 300$, $\beta = 3$, and the law of the environment is Uniform$[-1,1]$.}
\label{fig1}
\end{figure}

\section{Proof}\label{proofsec}
Let us begin by proving Proposition \ref{equivprop}.
\begin{proof}[Proof of Proposition \ref{equivprop}]
Let $P$ and $P'$ be independent random paths drawn from $\mu$. Note that by the independence of $P$ and $P'$,
\begin{align*}
\rho(\mu) &= \ee\biggl(\ee\biggl(\frac{|P\cap P'|}{n}\biggl|P\biggr)\biggr)= \ee(\ell(\mu, P))\le \ell(\mu). 
\end{align*}
Conversely, suppose that $p^*= (x_1^*,\ldots,x_n^*)$ is  path that maximizes $\ell(\mu,p)$. Let $P = (X_1,\ldots, X_k)$ be a random path drawn from $\mu$. Then
\begin{align*}
\ell(\mu) &= \ell(\mu, p^*)=\frac{1}{n}\sum_{k=1}^n \pp(X_k=x_k^*)\\
&\le \frac{1}{n}\sum_{k=1}^n \biggl(\sum_{x\in \zz^d} \pp(X_k=x)^2\biggr)^{1/2}\\
&\le \biggl(\frac{1}{n}\sum_{k=1}^n\sum_{x\in \zz^d} \pp(X_k=x)^2\biggr)^{1/2}= \sqrt{\rho(\mu)}.
\end{align*}
This completes the proof of the proposition.
\end{proof}
The next step is to record a few useful results about the law of the environment, under the assumptions made in Subsection \ref{resultsec}. The first lemma records two basic facts about the function $h$.
\begin{lmm}\label{hlmm}
The function $h$ defined in \eqref{hdef} is strictly positive everywhere in $(a,b)$ and uniformly bounded above.
\end{lmm}
\begin{proof}
It is not hard to see that $m\in (a,b)$. From this and the definition of $h$ it follows that if $x\in [m,b)$, then $h(x)>0$. On the other hand, by the definition of $m$,
\[
\int_a^b(y-m)f(y)dy=0.
\]
Thus, if $x\in (a,m]$, then
\[
h(x) = \frac{\int_a^x (m-y)f(y)dy}{f(x)},
\]
which is again strictly positive. The boundedness of $h$ follows simply by the boundedness of $|h'|$ assumed in \eqref{hcond} and the boundedness of the interval~$(a,b)$.
\end{proof}
The next lemma proves a crucial integration by parts formula for the probability density $f$.
\begin{lmm}\label{intlmm}
For any bounded differentiable function $g:(a,b)\to \rr$,
\[
\int_a^b (x-m) g(x)f(x)dx = \int_a^b h(x)g'(x)f(x)dx.
\]
\end{lmm}
\begin{proof}
Let 
\[
q(x) := h(x)f(x) = \int_x^b (y-m) f(y)dy = \int_a^x(m-y) f(y)dy. 
\]
By the dominated convergence theorem, it follows that $q(x)\to 0$ as $x\to b$ or $x\to a$. Therefore, since $g$ is bounded, integration by parts gives
\[
\int_a^b g'(x) q(x) dx = -\int_a^b g(x)q'(x)dx. 
\]
But $g'(x)q(x) = h(x) g'(x) f(x)$, and $g(x)q'(x) = -(x-m)g(x)f(x)$. This completes the proof. 
\end{proof}
The next lemma shows that the probability density $f$ satisfies what is known as a Poincar\'e inequality in the measure concentration literature.
\begin{lmm}\label{poin1}
If $X$ is a random variable with probability density $f$, then for any differentiable function $g:(a,b)\to \rr$ such that $g$ and $g'$ are uniformly bounded, we have 
\[
\var(g(X)) \le K\ee(g'(X)^2),
\]
where $K := \sup_{a<x<b} |h(x)|$. 
\end{lmm}
\begin{proof}
Note that
\begin{align*}
\var(g(X)) \le \ee(g(X)-g(m))^2 = \int_a^b (g(x)-g(m))^2 f(x) dx.
\end{align*}
Now, if $x\ge m$, then
\begin{align*}
(g(x)-g(m))^2 &= \biggl(\int_m^x g'(y) dy\biggr)^2\\
&\le (x-m) \int_m^x g'(y)^2 dy.
\end{align*}
Similarly, if $x\le m$, then 
\begin{align*}
(g(x)-g(m))^2 &= \biggl(\int_x^m g'(y)dy\biggr)^2\\
&\le (m-x) \int_x^m g'(y)^2dy.
\end{align*}
Putting together the last three displays, we get
\begin{align*}
\int_m^b (g(x)-g(m))^2f(x) dx &\le \int_m^b \int_m^x (x-m)g'(y)^2f(x) dydx\\
&= \int_m^b \int_y^b (x-m) g'(y)^2 f(x)dxdy\\
&= \int_m^b g'(y)^2 h(y) f(y) dy,
\end{align*}
and similarly,
\[
\int_a^m (g(x)-g(m))^2f(x) dx \le \int_a^m g'(y)^2 h(y) f(y)dy. 
\]
Combining, we have
\[
\int_a^b (g(x)-g(m))^2 f(x)dx \le \int_a^b g'(x)^2 h(x)f(x) dx. 
\]
This completes the proof of the lemma. 
\end{proof}
It is well known that if a measure satisfies a Poincar\'e inequality, its $n$-fold product also satisfies a Poincar\'e inequality with the same multiplicative constant. This is known as the tensorization property of Poincar\'e inequalities. The proof of the tensorization property is quite simple, so we present it here to save the curious reader the trouble of looking up the proof. 
\begin{lmm}\label{poinlmm}
Let $X_1,\ldots, X_n$ be i.i.d.~random variables with probability density $f$. Then for any differentiable function $g:(a,b)^n \to \rr$ such that $g$ and its partial derivatives are uniformly bounded, we have
\[
\var(g(X_1,\ldots,X_n)) \le K \sum_{i=1}^n \ee(\partial_i g(X_1,\ldots,X_n)^2),
\]
where $K = \sup_{a<x<b} |h(x)|$, and $\partial_i g$ is the partial derivative of $g$ in  coordinate $i$. 
\end{lmm}
\begin{proof}
Lemma \ref{poin1} proves the claim for $n=1$. Suppose that $n\ge 2$ and that the claim holds for $n-1$. Let 
\[
g_0(x_1,\ldots,x_{n-1}) := \int_a^b g(x_1,\ldots,x_n) f(x_n)dx_n.
\]
Then the boundedness of derivatives implies that for each $1\le i\le n-1$, 
\[
\partial_i g_0(x_1,\ldots,x_{n-1}) = \int_a^b \partial_i g(x_1,\ldots,x_n) f(x_n) dx_n,
\]
and hence
\begin{align}\label{induc}
\partial_i g_0(x_1,\ldots,x_{n-1})^2 \le \int_a^b \partial_i g(x_1,\ldots,x_n)^2 f(x_n) dx_n.
\end{align}
Now note that
\[
g_0(X_1,\ldots,X_{n-1}) = \ee(g(X_1,\ldots,X_n)|X_1,\ldots,X_{n-1}). 
\]
Thus, by the standard decomposition of variance as the sum of expected value of conditional variance and the variance of conditional expectation, we have 
\begin{align*}
\var(g(X_1,\ldots,X_n)) &= \ee(\var(g(X_1,\ldots,X_n)|X_1,\ldots,X_{n-1})) \\
&\qquad + \var(g_0(X_1,\ldots, X_{n-1})).
\end{align*}
But by the case $n=1$, 
\begin{align*}
&\var(g(X_1,\ldots,X_n)|X_1,\ldots,X_{n-1})\\
&\le K \ee(\partial_n g(X_1,\ldots,X_n)^2|X_1,\ldots,X_{n-1}),
\end{align*}
and hence
\begin{align*}
&\ee(\var(g(X_1,\ldots,X_n)|X_1,\ldots,X_{n-1}))\\
&\le K \ee(\partial_n g(X_1,\ldots,X_n)^2).
\end{align*}
On the other hand, by the case $n-1$, 
\[
\var(g_0(X_1,\ldots, X_{n-1})) \le K\sum_{i=1}^{n-1}\ee(\partial_i g_0(X_1,\ldots,X_{n-1})^2). 
\]
But by \eqref{induc},
\[
\ee(\partial_i g_0(X_1,\ldots,X_{n-1})^2) \le \ee(\partial_i g(X_1,\ldots,X_n)^2). 
\]
The proof is now completed by putting together the above estimates. 
\end{proof}

We are now ready to start the proof of Theorem \ref{mainthm}. Throughout the proof we will assume without loss that $\beta \ge 1$.   We will also assume without loss that the mean $m$ of $\nu$ equals $0$, so that $a<0$ and $b>0$. There is no loss in this second assumption because subtracting off a constant from the environment variables does not change the Gibbs measure.

Let us begin by defining some random variables for later use. For each $k\ge 1$ and $x\in \zz^d$, let $\theta_{k,x}$ be the probability that a path drawn from the Gibbs measure $\mu_{\beta,n}$ passes through $x$ at step $k$. For each $k$, define a random variable
\[
\alpha_k := \sum_{x} \theta_{k,x}^2.
\]
Note that $\alpha_k$  is the probability that two independent paths drawn from the Gibbs measure meet at step $k$. In particular, 
\[
\alpha_k \le \sum_x \theta_{k,x}=1,
\]
and since $\theta_{k,x} =0$ for any $x$ at a distance greater than $k$ from the origin, the Cauchy--Schwarz inequality shows that 
\begin{equation}\label{alphalow}
\alpha_k \ge \frac{(\sum_x \theta_{k,x})^2}{(2k+1)^d} = \frac{1}{(2k+1)^d}. 
\end{equation}
For each $k$, let $\mf_k$ be the $\sigma$-algebra generated by the random variables $\{\omega_{j,x}: j\ne k, x\in \zz^d\}$, and let $\alpha'_k := \ee(\alpha_k|\mf_k)$. The following lemma allows us to control $\alpha_k'$ in terms of $\alpha_k$. Recall that $\nu$ is the law of the environment.
\begin{lmm}\label{alphamainlmm}
There is a positive constant $L_1$ depending only on $\nu$, such that for any $1\le k\le n$, $\alpha_k' \le e^{L_1\beta} \alpha_k$. Explicitly, we can take $L_1=4(b-a)$. 
\end{lmm}
\begin{proof}
Fix some $k$. Consider the modified Gibbs measure obtained by replacing $\omega_{k,x}$ with $0$ for every $x$. Let $\zeta_{k,x}$ be the probability that a path chosen from this modified measure visits $x$ at step $k$. 

Let $\omega'$ denote the modified environment, that is,
\[
\omega'_{j,x} =
\begin{cases}
\omega_{j,x} &\text{ if } j\ne k,\\
0 &\text{ if } j=k.
\end{cases}
\]
Since $\nu$ is supported on $(a,b)$, it follows for any path $p= (x_1,\ldots,x_n)$, 
\[
e^{-\beta b}\exp\biggl(\beta\sum_{j=1}^n\omega_{j,x_j}\biggr)\le \exp\biggl(\beta\sum_{j=1}^n\omega'_{j,x_j}\biggr) \le e^{-\beta a}\exp\biggl(\beta\sum_{j=1}^n\omega_{j,x_j}\biggr)
\]
If $Z$ and $Z'$ are the normalizing constants for the two models, then summing the above expression over all $p$ gives 
\[
e^{-\beta b} Z\le Z'\le e^{-\beta a}Z. 
\]
Similarly, if $Z_{k,x}$ and $Z_{k,x}'$ are the sum of weights over all paths passing through $x$ at time $k$ in the two models, then 
\[
e^{-\beta b} Z_{k,x}\le Z'_{k,x}\le e^{-\beta a}Z_{k,x}.
\]
Since $\theta_{k,x} = Z_{k,x}/Z$ and $\zeta_{k,x} = Z_{k,x}'/Z'$, the last two displays imply that
\begin{equation*}
e^{-\beta(b-a)} \theta_{k,x}\le \zeta_{k,x} \le e^{\beta(b-a)} \theta_{k,x}. 
\end{equation*}
Also, note that $\zeta_{k,x}$ is an $\mf_k$-measurable random variable. Thus,
\begin{align*}
\alpha_k' &= \ee\biggl(\sum_x \theta_{k,x}^2 \biggl|\mf_k\biggr)\le \ee\biggl(\sum_x e^{2\beta(b-a)}\zeta_{k,x}^2 \biggl|\mf_k\biggr)\\
&= e^{2\beta(b-a)}\sum_x \zeta_{k,x}^2\le e^{4\beta(b-a)} \sum_x\theta_{k,x}^2 = e^{4\beta(b-a)} \alpha_k.
\end{align*}
This completes the proof of the lemma.
\end{proof}
Next, let
\[
 \gamma_k := \sum_x h(\omega_{k,x}) \theta_{k,x}, \ \ \gamma_k' := \ee(\gamma_k |\mf_k). 
\]
From here, the proof goes roughly as follows. Our goal is to show that with high probability, $\frac{1}{n}\sum\alpha_k$ is not small. By Lemma \ref{alphamainlmm}, it suffices to prove this for $\frac{1}{n}\sum \alpha_k'$ instead. This is done in four steps. 

The first step is to prove an inequality of the form $\gamma_k'\le C\alpha_k' + C/\beta$, where $C$ is a universal constant. This is proved using integration by parts. The second step is to show that if $\alpha_k'$ is small for some $k$, then with high probability, $\gamma_k$ has small fluctuations conditional on $\mf_k$, and hence  $\gamma_k \approx \gamma_k'$. This is established using the Poincar\'e inequality (Lemma \ref{poinlmm}), because calculations yield $\alpha_k'$ on the right side of the inequality. The third step is to show that with high probability, $\frac{1}{n}\sum \gamma_k$ cannot be smaller than a threshold that is independent of  $\beta$. This is proved using lower tail inequalities for sums of independent random variables. Finally, in the fourth step, we combine the first three steps to complete the proof, as follows: If $\frac{1}{n}\sum\alpha_k'$ is sufficiently small and $\beta$ is sufficiently large, then by the first step, we can make $\frac{1}{n}\sum \gamma_k'$  as small as we like. In particular, we can make it smaller than the threshold from the third step. But also, if $\frac{1}{n}\sum\alpha_k'$ is small, the second step dictates that $\frac{1}{n}\sum \gamma_k \approx \frac{1}{n}\sum \gamma_k'$. This yields a contradiction to the third step, completing the proof. 

The following lemma allows us to control $\gamma_k'$ in terms of $\alpha_k'$. 
\begin{lmm}\label{alphagammalmm}
There is a positive constant $L_2$ depending only on $\nu$ such that for any $1\le k\le n$,
\[
\gamma_k' \le L_2 \alpha_k' + \frac{L_2}{\beta}.
\]
\end{lmm}
\begin{proof}
Let
\[
\tau_k := \sum_x\omega_{k,x} \theta_{k,x}.
\]
It is not hard to check that 
\[
 \fpar{\theta_{k,x}}{\omega_{k,x}} = \beta \theta_{k,x} (1-\theta_{k,x}).
\]
Therefore, integration by parts (Lemma \ref{intlmm}) and the boundedness of $h$ (Lemma \ref{hlmm}) give
\begin{align*}
\ee(\tau_k|\mf_k) &= \sum_x \ee(\omega_{k,x} \theta_{k,x}|\mf_k)\\
&= \sum_x \ee\biggl(h(\omega_{k,x}) \fpar{\theta_{k,x}}{\omega_{k,x}}\biggl|\mf_k\biggr)\\
&= \beta \sum_x\ee(h(\omega_{k,x}) \theta_{k,x} (1-\theta_{k,x})|\mf_k)\\
&\ge \beta \gamma_k' - C\beta \alpha_k',
\end{align*}
where $C$ is a positive constant that depends only on $\nu$. 
On the other hand, since $\nu$ in supported on $(a,b)$, it follows that 
\[
\ee(\tau_k|\mf_k)\le b\ee\biggl(\sum_x \theta_{k,x}\biggl|\mf_k\biggr) = b. 
\]
The proof is completed by combining the two inequalities.
\end{proof}
Next, for each $\ep>0$, let
\begin{equation}\label{sdef}
S_\ep := \{1\le k\le n: \alpha_k'\le \ep\}.
\end{equation}
The following lemma gives a preliminary control on the size of $S_\ep$ using $\gamma_k$ and $\gamma_k'$. This will be used later to obtain a better control on the size of $S_\ep$. 
\begin{lmm}\label{gammabdlmm}
There is a positive constant $L_3$ depending only on $\nu$ such that for any $\ep>0$, 
\begin{align*}
\ee\biggl|\sum_{k\in S_\ep}(\gamma_k-\gamma_k')\biggr|&\le L_3 n \beta \sqrt{\ep}.
\end{align*}
\end{lmm}
\begin{proof}
Note that
\begin{align*}
\fpar{\gamma_k}{\omega_{k,x}} &= h'(\omega_{k,x})\theta_{k,x} + \sum_y h(\omega_{k,y}) \fpar{\theta_{k,y}}{\omega_{k,x}}\\
&= h'(\omega_{k,x})\theta_{k,x} + \beta h(\omega_{k,x}) \theta_{k,x} - \beta \sum_y h(\omega_{k,y}) \theta_{k,y}\theta_{k,x}.
\end{align*}
Consequently,  by assumption \eqref{hcond} and Lemma \ref{hlmm}, and the assumption that $\beta\ge 1$, we have 
\begin{align*}
\sum_x \biggl(\fpar{\gamma_k}{\omega_{k,x}}\biggr)^2 &\le C\beta^2 \sum_x \theta_{k,x}^2 = C\beta^2 \alpha_k,
\end{align*}
where $C$ depends only on $\nu$. 
Thus, by the Poincar\'e inequality (Lemma \ref{poinlmm}),
\begin{align*}
\var(\gamma_k|\mf_k) \le CK\beta^2 \alpha_k'.
\end{align*}
Take any $\ep >0$. Then by the above bound,
\begin{align*}
\sum_{k=1}^n \ee(|\gamma_k -\gamma_k'|1_{\{\alpha_k'\le \ep\}})&= \sum_{k=1}^n \ee(\ee(|\gamma_k-\gamma_k'||\mf_k) 1_{\{\alpha_k'\le \ep\}})\\
&\le \sum_{k=1}^n \ee((\var(\gamma_k|\mf_k))^{1/2} 1_{\{\alpha_k'\le \ep\}})\\
&\le n\beta\sqrt{CK\ep}.
\end{align*}
This completes the proof of the lemma. 
\end{proof}
Next, for each $A\subseteq \{1,\ldots,n\}$, let 
\[
\gamma(A) := \sum_{k\in A} \gamma_k. 
\]
The following lemma gives a lower tail bound for the minimum of $\gamma(A)/|A|$ over all $A$ with $|A|\ge n/2$. The important thing is that the bound has no dependence on $\beta$.
\begin{lmm}\label{kappalmm}
There is a positive constant $\kappa$ depending only on $\nu$ and $d$ such that if we define the event
\[
E :=\{\gamma(A)\le \kappa |A| \textup{ for some $A$ with $|A|\ge n/2$}\},
\]
then $\pp(E) \le e^{-n}$.
\end{lmm}
\begin{proof}
For $p = (x_1,\ldots, x_n)\in \cp_n$ and $A\subseteq \{1,\ldots,n\}$, define
\[
\psi(p,A) := \sum_{k\in A} h(\omega_{k,x_k}). 
\] 
Note that if $P$ is a random path drawn from the Gibbs measure $\mu_{\beta,n}$, then $\gamma(A)$ is simply the expected value of $\psi(P,A)$ under $\mu_{\beta,n}$. This shows that
\begin{align*}
\gamma(A) &\ge \min\{\psi(p,A): p\in \cp_n\}. 
\end{align*}
Thus, for any $\kappa>0$, 
\begin{align}
\pp(E) &\le \pp(\gamma(A)\le \kappa n \textup{ for some $A$ with $|A|\ge n/2$}) \nonumber\\
&\le \sum_{A: |A|\ge n/2} \pp(\gamma(A)\le \kappa n) \nonumber \\
&\le \sum_{A: |A|\ge n/2} \sum_{p\in \cp_n}\pp(\psi(p,A)\le \kappa n).\label{unionbd}
\end{align}
For each $\lambda >0$, let
\[
\phi(\lambda) := \int_a^be^{-\lambda h(x)}f(x)dx. 
\]
Since $h$ is strictly positive everywhere in $(a,b)$ (by Lemma \ref{hlmm}), it follows by the dominated convergence theorem that 
\begin{align}\label{philim}
\lim_{\lambda \to \infty} \phi(\lambda)=0. 
\end{align}
Take any $A$ with $|A|\ge n/2$ and any $p\in \cp_n$. Then for any  $\lambda>0$,
\begin{align*}
\pp(\psi(p,A) \le \kappa n) &= \pp(e^{-\lambda \psi(p,A)} \ge e^{-\lambda \kappa n})\\
&\le e^{\lambda \kappa n} \ee(e^{-\lambda \psi(p,A)})\\
&= e^{\lambda \kappa n} \phi(\lambda)^{|A|} \le  e^{\lambda \kappa n} \phi(\lambda)^{n/2}. 
\end{align*}
Therefore by \eqref{unionbd},
\begin{align*}
\pp(E) &\le 2^n (2d)^n  e^{\lambda \kappa n} \phi(\lambda)^{n/2}.
\end{align*}
By \eqref{philim} we may choose a positive number $\lambda$ so large that 
\[
\log\phi(\lambda) \le -2\log 2-2\log (2d)-4.
\]
The proof is now completed by choosing $\kappa=1/\lambda$.
\end{proof}
We are now ready to finish the proof of Theorem \ref{mainthm}.
\begin{proof}[Proof of Theorem \ref{mainthm}]
Recall the set $S_\ep$ defined in \eqref{sdef}. By Lemma \ref{alphagammalmm}, it follows that if $k\in S_\ep$, then
\[
\gamma_k' \le L_2\ep + \frac{L_2}{\beta}. 
\]
Thus, there exist positive constants $\ep_0$ and $\beta_0$ depending only on $\nu$ and $d$, such that if $\ep\le\ep_0$ and $\beta \ge \beta_0$, then for each $k\in S_\ep$,
\begin{equation}\label{kbd}
\gamma_k'\le \frac{\kappa}{2},
\end{equation}
where $\kappa$ is the constant from Lemma \ref{kappalmm}. We will henceforth assume that $\ep\le \ep_0$ and $\beta \ge \beta_0$. Let $E$ be the event from Lemma~\ref{kappalmm} and let $E^c$ denote its complement. Define another event
\[
F := \{|S_\ep|\ge n/2\}.
\]
If the event $F\cap E^c$ happens, then
\[
\gamma(S_\ep)> \kappa|S_\ep|.
\]
Moreover, by \eqref{kbd}, 
\[
\sum_{k\in S_\ep}\gamma_k'\le \frac{\kappa|S_\ep|}{2}. 
\]
Consequently, if the event $F\cap E^c$ happens, then 
\[
\biggl|\sum_{k\in S_\ep}(\gamma_k-\gamma_k')\biggr|> \frac{\kappa|S_\ep|}{2}\ge \frac{\kappa n}{4}.
\]
Thus, by Lemma \ref{gammabdlmm}, 
\begin{align*}
\pp(F\cap E^c) &\le \pp\biggl(\biggl|\sum_{k\in S_\ep}(\gamma_k-\gamma_k')\biggr|\ge \frac{\kappa n}{4}\biggr)\\
&\le \frac{4}{\kappa n} \ee\biggl|\sum_{k\in S_\ep}(\gamma_k-\gamma_k')\biggr|\\
&\le \frac{4L_3\beta\sqrt{\ep}}{\kappa}. 
\end{align*}
Therefore by Lemma \ref{kappalmm},
\begin{align}
\pp(|S_\ep|\ge n/2) &= \pp(F) \le \pp(F\cap E^c) + \pp(E)\nonumber\\
&\le C_1\beta\sqrt{\ep} + e^{-n},\label{sbd}
\end{align}
where $C_1$ depends only on $\nu$ and $d$. 
Let \[
R_\ep := \{1,\ldots,n\}\setminus S_\ep.
\]
If $k\in R_\ep$, then $\alpha_k'> \ep$, and hence by Lemma \ref{alphamainlmm}, $\alpha_k > e^{-L_1\beta} \ep$.  
Therefore, if $|S_\ep|< n/2$, then
\begin{align*}
\sum_k \alpha_k &\ge \sum_{k\in R_\ep} \alpha_k > |R_\ep| e^{-L_1\beta}\ep\\
&> \frac{1}{2} ne^{-L_1\beta}\ep.
\end{align*}
Thus, by \eqref{sbd},
\begin{align*}
\pp\biggl(\sum_k\alpha_k \le\frac{1}{2} ne^{-L_1\beta}\ep\biggr) &\le \pp(|S_\ep|\ge n/2)\\
 &\le C_1\beta\sqrt{\ep} + e^{-n}.
\end{align*}
But it is not hard to see that 
\[
\frac{1}{n}\sum_{k=1}^n \alpha_k = \rho(\mu_{\beta,n}). 
\]
Thus,
\begin{align*}
\pp\biggl(\rho(\mu_{\beta,n}) \le \frac{1}{2} e^{-L_1\beta}\ep\biggr)  &\le C_1\beta\sqrt{\ep} + e^{-n}.
\end{align*}
Writing $\delta =  \frac{1}{2} e^{-L_1\beta}\ep$, it follows from this  that there are positive constants $\delta_0$ and $C_2$, depending only on $\nu$ and $d$, such that for any $\delta \le \delta_0$,
\begin{align}\label{rho1}
\pp(\rho(\mu_{\beta,n}) \le \delta)  \le e^{C_2\beta}\sqrt{\delta} + e^{-n}.
\end{align}
By \eqref{alphalow}, we know that
\[
\rho(\mu_{\beta,n}) \ge \frac{1}{n}\sum_{k=1}^n \frac{1}{(2k+1)^d}\ge \frac{1}{3^dn}.
\]
Therefore the left side of \eqref{rho1} is zero if $\delta < 1/(3^d n)$. Consequently, if $\delta < 1/(3^d n)$, then
\[
\pp(\rho(\mu_{\beta,n}) \le \delta) = 0 \le e^{C_2\beta}\sqrt{\delta},
\]
whereas if $\delta \ge 1/(3^d n)$, then 
\begin{align*}
\pp(\rho(\mu_{\beta,n})\le \delta) &\le e^{C_2\beta}\sqrt{\delta} + e^{-n}\\
&\le  e^{C_2\beta}\sqrt{\delta} + e^{-1/(3^d \delta)}\le (e^{C_2\beta} + C_3) \sqrt{\delta},
\end{align*}
where $C_3$ depends only on $d$. 
Since $\beta \ge \beta_0>0$, and $\beta_0$ depends only on $d$ and $\nu$, it follows that $C_3 \le e^{C_4 \beta}$ for some $C_4$ depending on $d$ and $\nu$. Thus, irrespective of the value of $\delta$, $\pp(\rho(\mu_{\beta,n})\le \delta)\le e^{C_5\beta} \sqrt{\delta}$ for some $C_5$ depending only on $d$ and $\nu$. The tail bound for $\ell(\mu_{\beta,n})$ is now derived using Proposition \ref{equivprop}.
\end{proof}

\section*{Acknowledgments}
I thank Erik Bates, Francis Comets, and the anonymous referees for a number of useful comments and suggestions.

\end{document}